\documentclass[11pt]{article}

\usepackage{amsmath,amsthm,amssymb}
\usepackage{booktabs}
\usepackage{url}
\usepackage{graphicx}
\usepackage[T1]{fontenc}
\usepackage[utf8]{inputenc}
\usepackage[numbers]{natbib}

\usepackage{enumitem}
\setlist[enumerate]{label=\roman*),itemsep=0pt}

\usepackage{array}
\newcolumntype{P}[1]{>{\centering\arraybackslash$}p{#1}<{$}}

\usepackage{calc}
\newtheorem{thm}{Theorem}
\newtheorem{lem}[thm]{Lemma}
\newtheorem{conj}[thm]{Conjecture}
\newtheorem{obs}[thm]{Observation}

\newtheorem{prop}[thm]{Proposition}

\theoremstyle{definition}
\newtheorem{df}[thm]{Definition}

\theoremstyle{remark}
\newtheorem{ex}[thm]{Example}
\newtheorem{rem}[thm]{Remark}

\newcommand{\bu}{\ensuremath{\boldsymbol{u}}}

\newcommand{\N}{\mathbb{N}}

\newcommand{\Lc}{\mathcal{L}}
\newcommand{\Pc}{\mathcal{P}}
\newcommand{\Ec}{\mathcal {E}\text{-}\Pc}
\newcommand{\Aj}{\mathcal{A}_1}
\newcommand{\Ad}{\mathcal{A}_2}
\newcommand{\Er}{\mathrm{E}}
\newcommand{\Rr}{\mathrm{R}}

\newcommand{\pr}{\mathfrak{p}}
\newcommand{\su}{\mathfrak{s}}
\newcommand{\fa}{\mathfrak{w}}

\begin{document}

 \title{Morphisms generating antipalindromic words}

 \author{Petr Ambro\v{z}, Zuzana Mas\'akov\'a, Edita Pelantov\'a \\[3mm] 
    \small Faculty of Nuclear Sciences and Physical Engineering \\
    \small Czech Technical University in Prague\\ 
    \small Trojanova 13, 120 00 Praha 2, Czech Republic}

 \date{}
 
  \maketitle


 \begin{abstract}
 We introduce two classes of morphisms over the alphabet $A=\{0,1\}$ whose fixed points contain
 infinitely many antipalindromic factors. An antipalindrome is a finite word invariant under the
 action of the antimorphism $\Er:\{0,1\}^*\to\{0,1\}^*$, defined by $\Er(w_1\cdots
 w_n)=(1-w_{n})\cdots(1-w_1)$. We conjecture that these two classes contain all morphisms (up to
 conjugation) which generate infinite words with infinitely many antipalindromes. This is an
 analogue to the famous HKS conjecture concerning infinite words containing infinitely many
 palindromes. We prove our conjecture for two special classes of morphisms, namely (i) uniform
 morphisms and (ii) morphisms with fixed points containing also infinitely many palindromes. \\
 
 \medskip
 
\end{abstract}


 \section{Introduction}

Palindromic words are infinite words over a finite alphabet $A$ which contain arbitrarily long
palindromes. Recall that a palindrome is a finite word $w$ which is read the same backwards and
forwards, i.e., $w=w_1w_2\cdots w_n=w_nw_{n-1}\cdots w_1$. Palindromic words have been extensively
studied since the observation of Hof, Knill and Simon~\cite{hof-knill-simon-cmp-174} that they can
be used for construction of aperiodic potentials of discrete Schr\"odinger operators with purely
singular continuous spectrum. Such Schr\"odinger operators seem to describe well the behaviour of
one-dimensional structures known under the name quasicrystals.

A large class of palindromic words is the family of Sturmian words defined as infinite aperiodic
words with minimal complexity.  One of Sturmian words is the Fibonacci word
$\boldsymbol{f}=010010100100101001\cdots$. The word $\boldsymbol{f}$ can be constructed by iterating
the rewriting rule $0\mapsto 01$, $1\mapsto0$, i.e.,
$$
0\mapsto01\mapsto010\mapsto01001\mapsto01001010\mapsto\cdots
$$
Note that the word in $i$-th iteration is a prefix of the word in iteration $i+1$ and the
infinite word $\boldsymbol{f}$ is defined naturally. The construction can be formalized using the
notion of a homomorphism over the free monoid $A^*$ of all words over a finite alphabet (equipped
with the operation of concatenation and the empty word as the neutral element).  In the context of
combinatorics on words, the homomorphisms are called just morphisms. The Fibonacci word
$\boldsymbol{f}$ is thus an example of an infinite word fixed by the morphism
$\varphi:\{0,1\}^*\to\{0,1\}^*$, defined by $\varphi(0)=01$, $\varphi(1)=0$.

The authors of~\cite{hof-knill-simon-cmp-174} conjecture that any palindromic fixed point of a
morphism can be generated by a morphism conjugated to an element of the so-called class $\Pc$ -- a
family of morphisms in a special form, namely that $\varphi(a)=p_aq$ where $q$ and $p_a$ for $a\in
A$ are palindromes.  This -- the so-called HKS conjecture -- has been proven for binary words by Tan
in~\cite{tan-tcs-389}. Partial results for infinite words over larger alphabets have been also
given. For example, the HKS conjecture is proved in~\cite{allouche-baake-cassaigne-damanik-tcs-292}
for periodic words, in \cite{labbe-pelantova-ejc-51} for fixed points of marked morphisms, and in
\cite{masakova-pelantova-starosta-ejc-62} for words coding non-degenerated exchange of three
intervals. In~\cite{harju-vesti-zamboni-monatshefte-2016}, a modified version of HKS conjecture has
been proven for rich words.

The definition of a palindrome can be formulated using the notion of an antimorphism over the monoid
$A^*$. A mapping $\eta:A^*\to A^*$ is an antimorphism if $\eta(vw)=\eta(w)\eta(v)$ for any pair of
words $v,w\in A^*$. A palindrome is a finite word invariant under the mirror image antimorphism
$\Rr$. Words invariant under other involutive antimorphisms are called generalized palindromes or
pseudopalindromes. In the particular case of a binary alphabet, the only involutive antimorphism
other than $\Rr$ is the exchange map $\Er$. The words $w$ such that $\Er(w)=w$ are called
antipalindromes. For example, the shortest nonempty antipalindrome is $01$, as $E(01)=E(1)E(0)=01$.
An infinite word containing infinitely many such antipalindromes is called antipalindromic.  The
well known Thue-Morse word $\boldsymbol{t}=0110100110010110\cdots$, both fixed points of the
morphism $\Theta:\ 0\mapsto 01,\ 1\mapsto10$, can serve as an example.  A large class of
antipalindromic words is given by complementary symmetric Rote words,
see~\cite{blondin-masse-brlek-labbe-vuillon-tcs-412}. These words are, however, not fixed by any
non-identical morphism, see~\cite{medkova-pelantova-vuillon-preprint}.

Our aim is to study a modification of the HKS conjecture to the case of antipalindromes. We define
two classes of morphisms $\Aj$, $\Ad$ such that any fixed point of a morphism in any of these
classes is antipalindromic. We conjecture that classes $\Aj$, $\Ad$ contain (up to conjugacy) all
primitive morphisms with antipalindromic fixed points.  The conjecture is supported by our results
formulated as Theorem~\ref{thm:mainUnifrom} and Theorem~\ref{thm:mainNonuniform} which state that
\begin{itemize}
\item
  if a binary uniform morphism $\varphi$ (i.e., such that the words $\varphi(0)$, $\varphi(1)$ are
  of the same length) has an antipalindromic fixed point, then $\varphi$ or $\varphi^2$ is (up to
  conjugacy) equal to a morphism in class $\Aj$;
\item
  if a morphism $\varphi$ is primitive and its fixed point contains infinitely many of both
  palindromes and antipalindromes, then $\varphi$ or $\varphi^2$ belongs to $\Aj\cap\Pc$ or
  $\Ad\cap\Pc$.
\end{itemize}

The situation can be summarized in a diagram displayed in Figure~\ref{f} which shows intersection of
the above mentioned classes $\Pc$, $\Aj$, $\Ad$.
\begin{figure}[!htp]
  \centering
  \includegraphics[scale=1]{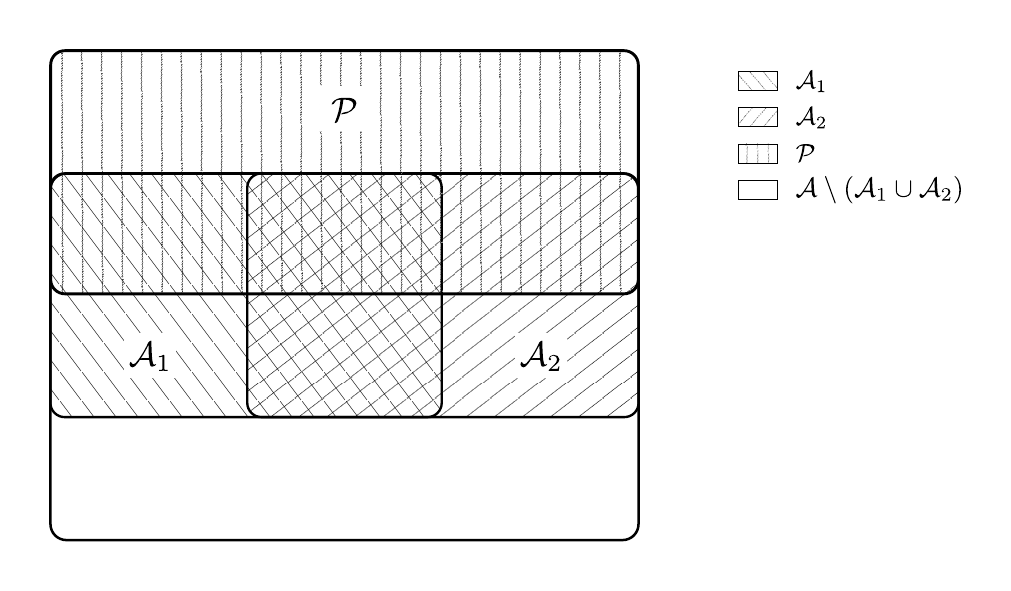}
  \caption{Relation of classes $\Aj$, $\Ad$ and $\Pc$. Denoting by ${\mathcal A}$ the set of all
    primitive morphisms with antipalindromic fixed points, we conjecture that the set ${\mathcal
      A}\setminus(\Aj\cup\Ad)$ is empty.}\label{f}
\end{figure}

This paper is organized as follows. In Section~\ref{sec:prelim} we recall a number of necessary
notions and tools from combinatorics on words. In particular, we recall several results on equations
on words and on the structure of bispecial factors in languages of morphic words.  In
Section~\ref{sec:Type12} we define classes $\Aj$, $\Ad$ and show that their fixed points are
antipalindromic. We also characterize morphisms in classes $\Pc\cap\Aj$ and $\Pc\cap\Ad$.
Section~\ref{sec:periodic} describes eventually periodic words which contain infinitely many
antipalindromes and those which contain infinitely many of both palindromes and
antipalindromes. Properties of languages of infinite antipalindromic aperiodic words are studied in
Section~\ref{sec:language}. The main results (Theorems~\ref{thm:mainUnifrom}
and~\ref{thm:mainNonuniform}) are stated in Section~\ref{sec:main}. Theorem~\ref{thm:mainUnifrom}
for uniform mosphism is also proved there. The proof of Theorem~\ref{thm:mainNonuniform} which
concerns non-uniform morphisms is very technical and requires a number of auxiliary results. They
are collected in Section~\ref{sec:proof}. Finally, we include comments and open problems.


\section{Preliminaries}\label{sec:prelim}

\subsection{Finite Words}\label{sec:fin_words}

Let $A$ be a finite set called \emph{alphabet}, its elements are called \emph{letters}. A finite
sequence $w=w_1\cdots w_n$ such that $w_i\in A$ for every $i=1,2,\ldots,n$ is called a \emph{word} over
$A$. The \emph{length} (the number of elements) of $w=w_1\cdots w_n$ is denoted by $|w|=n$. The
notation $|w|_a$ is used for the number of occurrences of the letter $a$ in $w$. The word of length
zero -- the so called \emph{empty word} -- is denoted by $\varepsilon$. The set of all finite words
over $A$ equipped with the operation concatenation of words forms the free monoid $A^*$, its neutral
element is $\varepsilon$.

Let $w=pus$ for some $p,u,s\in A^*$. Then $u$, $p$, and $s$ are called a \emph{factor}, a
\emph{prefix}, and a \emph{suffix} of $w$, respectively. Let $w\in A^*$ and $r\in\N$ then $w^r$
denotes the \emph{$r$-th power} 
of $w$, i.e., $w^r=\underbrace{ww\cdots w}_{\text{$r$-times}}$. A word $v\in A^*$ is called
\emph{primitive} if for each $w\in A^*$ and $r\in\N$ the equality $v=w^r$ implies that $w=v$ and
$r=1$.

A mapping $\varphi:A^*\rightarrow A^*$ is called a \emph{morphism} (over $A$) if
$\varphi(uv)=\varphi(u)\varphi(v)$ for every $u,v\in A^*$; it is called an \emph{antimorphism} if
$\varphi(uv)=\varphi(v)\varphi(u)$ for every $u,v\in A^*$. Obviously, both morphism and antimorphism
are fully defined by providing $\varphi(a)$ for all $a\in A$.

A morphism $\varphi:A^*\to A^*$ is said to be \emph{primitive} if there exists $k\in\N$ such that for
every pair $a,b\in A$ the letter $a$ occurs in the word $\varphi^k(b)$. A morphism $\varphi:A^*\to A^*$
is called \emph{uniform} if for  every pair $a,b\in A$ we have $|\varphi(a)|=|\varphi(b)|$.

Throughout the paper, we use two important antimorphisms $\Rr$ and $\Er$. The antimorphism
$\Rr:A^*\rightarrow A^*$, called \emph{mirror image map}, is defined as $\Rr(a)=a$ for every $a\in
A$, that is, $\Rr(w_1\cdots w_n)=w_n\cdots w_1$. The antimorphism
$\Er:\{0,1\}^*\rightarrow\{0,1\}^*$ is the \emph{exchange map} defined by $\Er(0)=1$ and $\Er(1)=0$.
Note that both antimorphisms are involutions, i.e., $\Rr^2=\Er^2={\rm id}$.  Let $w$ be a word, if
$\Rr(w)=w$ then $w$ is said to be a \emph{palindrome}, if $\Er(w)=w$ then $w$ is called
\emph{antipalindrome}. Note that an antipalindrome necessarily has even length. The only word which
is both palindrome and antipalindrome is the empty word $\varepsilon$.

One can easily check the following properties of antimorphisms $\Rr$, $\Er$ and Thue-Morse morphism
$\Theta$ given by $\Theta(0)=01$ and $\Theta(1)=10$.

\begin{obs}\label{obs}
  Let $w\in\{0,1\}^*$. Then
  \begin{enumerate}
  \item
    w is a palindrome $\Leftrightarrow$ $\Er(w)$ is a palindrome,
  \item
    w is an antipalindrome $\Leftrightarrow$ $\Er(w)$ is an antipalindrome,
  \item
    $\Theta\Rr=\Er\Theta$,
  \item
    $\Theta(w)$ is a palindrome $\Leftrightarrow$ $w$ is an antipalindrome,
  \item
    $\Theta(w)$ is an antipalindrome $\Leftrightarrow$ $w$ is a palindrome.
  \end{enumerate}
\end{obs}

In our considerations we will repeatedly use several known results on the solutions of equations on
words. The statements of these results are summarized in the following proposition, the proofs can be
found in~\cite{lothaire1,brlek-hamel-nivat-reutenauer-ijfcs-15,pelantova-statosta-dmtcs-18}.

\begin{prop}[\cite{lothaire1,brlek-hamel-nivat-reutenauer-ijfcs-15,pelantova-statosta-dmtcs-18}]
  \label{p:properties}
  Let $x,y,z\in A^*$.
  \begin{enumerate}
  \item
    If $xy=yx$, then there exist $u\in A^*$ and $i,j\in\N$ such that $x=u^i$ and $y=u^j$.
  \item
    If $xy=yz$ and $x\neq\varepsilon$, then there exist $u,v\in A^*$ and $i\in\N$ such that $x=uv$,
    $y=(uv)^iu$, and $z=vu$.
  \item
    If $x,y\in\{0,1\}^*$ are nonempty palindromes and $xy$ is an antipalindrome, then there exist a
    palindrome $u\in\{0,1\}^*$ and $i,j\in\N$ such that $x=\big(u\Er(u)\big)^iu$ and
    $y=\big(\Er(u)u\big)^j\Er(u)$.
  \end{enumerate}
\end{prop}

Finally, let us recall the Fine-Wilf theorem.

\begin{thm}[\cite{lothaire1}]\label{thm:fine-wilf}
  Let $x,y\in A^*$. If $w$ is a prefix of both $x^r$ and $y^r$ for some $r\in\N$ and if
  $|w|\geq|x|+|y|-\gcd\{|x|,|y|\}$, then there is $z\in A^*$ such that $x=z^i$ and $y=z^j$ for some
  $i,j\in\N$.
\end{thm}


\subsection{Infinite words}

An \emph{infinite word} over an alphabet $A$ is an infinite sequence $\bu=u_0u_1u_2\cdots$ of
letters from $A$ (i.e., $u_i\in A$ for every $i\in\N$). The set of all infinite words over $A$ is
denoted $A^{\N}$. A finite word $w\in A^*$ of length $|w|=n$ is called a \emph{factor} of $\bu$ if
there is an index $i\in\N$ such that $w=u_iu_{i+1}\cdots u_{i+n-1}$. The index $i$ is called an
\emph{occurrence} of $w$ in $\bu$. The set of all factors of $\bu$ is called the \emph{language} of
$\bu$, denoted $\Lc(\bu)$.  We say that $\Lc(\bu)$ is closed under $\Rr$, if $w\in\Lc(\bu)$ implies
$\Rr(w)\in\Lc(\bu)$. Analogously, $\Lc(\bu)$ is closed under $\Er$ if $w\in\Lc(\bu)$ gives
$\Er(w)\in\Lc(\bu)$.

An infinite word $\bu$ is called \emph{eventually periodic} if there exist $v,w\in A^*$,
$w\neq\varepsilon$ such that $\bu=vw^{\infty}$, where $w^{\infty}$ denotes an infinite repetition of
$w$. If, moreover, $v=\varepsilon$, then $\bu$ is \emph{purely periodic}. An infinite word which is
not eventually periodic is called \emph{aperiodic}. An infinite word $\bu$ is called
\emph{recurrent} if each factor $w\in\Lc(\bu)$ has an infinite number of occurrences in $\bu$. If
there is a number $r$ for all $n\in\N$ such that each factor of $\bu$ of length $n$ has at least one
occurrence in the set $\{k,k+1,\ldots,k+r-n\}$ for every $k\in\N$, then $\bu$ is called
\emph{uniformly recurrent} and the mapping $n\mapsto r(n)$, where $r(n)$ is the minimal $r$ with the
above property, is called the \emph{recurrence function} of $\bu$.

The domain of a morphism $\varphi:A^*\to A^*$ can be naturally extended to infinite words by
$\varphi(\bu)=\varphi(u_0u_1u_2\cdots)=\varphi(u_0)\varphi(u_1)\varphi(u_2)\cdots$. If
$\varphi(\bu)=\bu$, then $\bu$ is called a \emph{fixed point} of $\varphi$.

A morphism $\varphi$ is called a \emph{substitution} if it has the following property: there are
$a\in A$ and $w\in A^*$, $w\neq\varepsilon$ such that $\varphi(a)=aw$ and $|\varphi^n(a)|$ tends to
infinity with growing $n$. Obviously, a substitution $\varphi$ has at least one fixed point, namely
$\bu=aw\varphi(w)\varphi^2(w)\cdots$.

Let $\varphi,\psi$ be morphisms over $A$. We say that $\varphi$ is \emph{left conjugate} to $\psi$
(or equivalently that $\psi$ is \emph{right conjugate} to $\varphi$) if there is $q\in A^*$ such
that $q\varphi(w)=\psi(w)q$ for every $w\in A^*$. In such a case the word $q$ is called the
\emph{conjugacy word}. If, moreover, $\varphi$ is the only left conjugate to itself, we say that
$\varphi$ is the \emph{leftmost conjugate} to $\psi$, denoted by $\psi_L$. The \emph{rightmost
  conjugate} to a morphism $\psi$, denoted by $\psi_R$, is defined analogously.

\begin{ex}
  Let $\psi:\{a,b\}^*\to\{a,b\}^*$ be defined by $\psi(a)=abaab$ and $\psi(b)=ab$. Then
  $\psi_L(a)=ababa$, $\psi_L(b)=ba$, $\psi_R(a)=ababa$, and $\psi_R(b)=ab$. Clearly $\psi_L$ and
  $\psi_R$ are also conjugate morphisms and their conjugacy word is equal to $q=ababa$.
\end{ex}

If a morphism $\psi$ is a conjugate to itself via a nonempty conjugacy word $q$ then $\psi$ is
called \emph{cyclic morphism} and it has a unique fixed point, namely $q^{\infty}$. Otherwise,
$\psi$ is called \emph{acyclic}. Any acyclic morphism has a leftmost and a rightmost conjugate.

Let $\mathrm{fst}(w)$ and $\mathrm{lst}(w)$ denote the first and the last letter of $w$, respectively.
Let $\psi$ be an acyclic morphism over a binary alphabet then obviously
\begin{equation}\label{eq:def_marked}
  \begin{split}
    & \mathrm{fst}(\psi_L(a))\neq\mathrm{fst}(\psi_L(b)) \qquad \text{if } a\neq b, \\
    & \mathrm{lst}(\psi_R(a))\neq\mathrm{lst}(\psi_R(b)) \qquad \text{if } a\neq b.
  \end{split}
\end{equation}
A morphism over an arbitrary alphabet $A$ satisfying~(\ref{eq:def_marked}) for all $a,b\in A$,
$a\neq b$ is called \emph{marked}. Thus a binary acyclic morphism is marked.

The following proposition summarizes several important properties of fixed points of primitive
morphisms, for proofs see~\cite{fogg,queffelec-lncs-1294}.

\begin{prop}
  Let $\varphi:A^*\to A^*$ be a primitive morphism and let $\bu$ be its fixed point. Then
  \begin{enumerate}
  \item
    $\bu$ is uniformly recurrent;
  \item
    if $\psi$ is a conjugate to $\varphi$ then $\psi$ is primitive;
  \item
    if $\psi$ is a conjugate to $\varphi$ and $\boldsymbol{v}$ is a fixed point of $\psi$ then
    $\Lc(\boldsymbol{v})=\Lc(\bu)$;
  \item
    for each $a\in A$ the following limit, called uniform frequency of the letter $a$ in $\bu$, exists
    \[
    \rho_a := \lim_{\substack{|w|\to\infty\\w\in\Lc(\bu)}}\frac{|w|_a}{|w|}.
    \]
  \end{enumerate}
\end{prop}

The \emph{stabilizer}~\cite{krieger-tcs-400} of an infinite word $\bu\in A^{\N}$ is the set
\[
\mathrm{Stab}(\bu) = \{ \varphi \text{ a morphism over } A : \varphi(\bu)=\bu \}.
\]
Clearly, $\mathrm{Stab}(\bu)$ is closed under the composition of morphisms and the identity morphism
belongs to $\mathrm{Stab}(\bu)$, i.e., $\mathrm{Stab}(\bu)$ is a monoid. An infinite word $\bu$ is
called \emph{rigid} if there exists a morphism $\varphi:A^*\to A^*$ such that
$\mathrm{Stab}(\bu)=\{\varphi^k:k\in\N\}$. In general, it is a difficult task to find the stabilizer
of an infinite word. Examples of classes of words known to be rigid are Sturmian
words~\cite{seebold-tcs-195}, Prouhet words~\cite{seebold-jalc-7}, and fixed points of epistandard
morphisms~\cite{krieger-tcs-400}.

An infinite word $\bu$ is called \emph{palindromic} if $\Lc(\bu)$ contains
an infinite number of palindromes. If $\bu$ is a uniformly
recurrent word which is palindromic, then $\Lc(\bu)$ is closed under $\Rr$.
Similarly, an infinite word $\bu$ is called \emph{antipalindromic} if $\Lc(\bu)$ contains
an infinite number of antipalindromes. The language $\Lc(\bu)$ of a uniformly
recurrent antipalindromic word $\bu$  is closed under $\Er$.

Tan~\cite{tan-tcs-389} proved that a fixed point of a primitive binary morphism $\varphi$ is
palindromic if and only if $\varphi$ or $\varphi^2$ is conjugate to a morphism in the so-called
class $\Pc$.

\begin{df}
  A primitive morphism $\psi:A^*\to A^*$ belongs to class $\Pc$ if there is a palindrome $p\in A^*$
  such that for each $a\in A$
  \[
    \psi(a) = pq_a, \qquad \text{where $q_a\in A^*$ is a palindrome.}
  \]
\end{df}

One can check whether a morphism $\psi$ belongs to class $\Pc$ by means of the morphism assigning to
any letter $a$ the reversal of the word $\psi(a)$.  The verification is based on the following
proposition, which has been proved in~\cite{tan-tcs-389} for binary morphisms and
in~\cite{labbe-pelantova-ejc-51} for morphisms over multilateral alphabets.

\begin{prop}[\cite{tan-tcs-389,labbe-pelantova-ejc-51}]\label{p:inverse_morphism}
  Let $\psi$ be a binary acyclic morphism. Then $\psi$ is conjugate to a morphism in class
  $\Pc$ if and only if $\Rr(\psi_R(a))=\psi_L(a)$ for $a\in\{0,1\}$.
\end{prop}


\subsection{Special factors}

Let $\bu$ be an infinite word over $A$. A factor $w\in\Lc(\bu)$ is called \emph{right special} if
there exist two different letters $a,b\in A$ such that $wa,wb\in\Lc(\bu)$. Analogically, $w$ is
called \emph{left special} if there exist two different letters $c,d\in A$ such that
$cw,dw\in\Lc(\bu)$. A factor $w$ is called \emph{bispecial} if it is both left and right special.
An infinite word $\bu$ is aperiodic if for every $n\in\N$ there is a left special factor of length $n$
and a right special factor of length $n$ in $\Lc(\bu)$.

Bispecial factors in fixed points of morphisms were described by Klouda~\cite{klouda-tcs-445} for a
broad class of morphisms.  The corollaries of said description for marked morphisms were formulated
by Labb\'e and Pelantov\'a~\cite[Proposition 28]{labbe-pelantova-ejc-51}. We give here a simplified
version for binary morphisms.

\begin{thm}\label{thm:klouda}
  Let $\varphi$ be a primitive binary morphism with an aperiodic fixed point $\bu$. Let $\varphi_L$ and
  $\varphi_R$ be the leftmost and the rightmost conjugate to $\varphi$, respectively, and let $q$ be
  their conjugacy word, i.e., $\varphi_R(a)q=q\varphi_L(a)$ for $a=0,1$. Then
  \begin{enumerate}
  \item
    For each bispecial factor $w\in\Lc(\bu)$ the word $\Phi(w):=\varphi_R(w)q$ is also a bispecial
    factor in $\Lc(\bu)$.
  \item
    There is a finite set of bispecial factors -- called initial bispecial factors -- such that
    every bispecial factor in $\Lc(\bu)$ is equal to $\Phi^k(w)$ for some initial bispecial $w$
    and some $k\in\N$.
  \end{enumerate}
\end{thm}


\section{Antipalindromic morphisms}\label{sec:Type12}

The problem of antipalindromic fixed points of morphisms was already studied by Labb\'e
in~\cite{labbe-memoire}. He defines the so-called class $\Ec$ as the set of binary morphisms of the
form
\begin{equation}\label{eq:classEP}
\varphi(a)=pp_a,\quad a=0,1,\quad\text{  where $p$, $p_0$, $p_1$ are antipalindromes}.
\end{equation}
Class $\Ec$ is a direct analogy of class $\Pc$. Morphisms in $\Ec$ do not necessarily generate
antipalindromic fixed points (as can be seen in Example~\ref{ex:EP}), nevertheless, class $\Ec$ is
useful in the formulation of the problem.

Here we define two classes of morphisms $\Aj$, $\Ad$ and show that they both generate
antipalindromic fixed points. We explain their relation to class $\Ec$.


\subsection{Class $\Aj$ -- uniform morphisms}

The first class contains uniform morphisms.


\begin{df}\label{df:type1}
  A morphism $\varphi:\{0,1\}^*\rightarrow\{0,1\}^*$ belongs to class $\Aj$ if there exist words
  $\pr,\su\in\{0,1\}^*$ such that $\pr\neq\varepsilon$, $\su$ is an antipalindrome, and
  \[
  \varphi(0) = \pr\su, \qquad \varphi(1) = \Er(\pr)\su.
  \]
\end{df}

Note that morphisms in class $\Aj$ are primitive, except the case $\varphi(0)=0^k$, $\varphi(1)=1^k$,
which produces trivial fixed points $0^\infty$, $1^\infty$. A primitive morphism $\varphi$ in class
$\Aj$ satisfies $|\varphi(w)|>|w|$ for every finite word $w\in\{0,1\}^*$.

\begin{rem}\label{rem:EPA1}
Stated in our notation, Labb\'e~\cite[Lemme~3.21]{labbe-memoire} shows that a uniform morphism 
$\varphi$ is in class $\Aj$ if and only if $\varphi\Theta$ belongs to class $\Ec$.
\end{rem}

Labb\'e also shows that morphisms in class $\Aj$ have antipalindromic fixed points. We include this
result and its demonstration for self-consistence.

\begin{prop}\label{thm:type1_is_antipalindromic}
  Let $\varphi$ be a primitive morphism in class $\Aj$ and let $\bu$ be its fixed point.  Then
  $\Lc(\bu)$ contains infinitely many antipalindromes.
\end{prop}

First we state a simple lemma.

\begin{lem}\label{lem:type1_and_E}
  Let $\varphi$ be a morphism in class $\Aj$ and let $w\in\{0,1\}^*$. Then
  \begin{equation}\label{eq:1}
    \Er(\su\varphi(w)) = \su\varphi(\Er(w)),
  \end{equation}
  where $\su$ is the suffix of $\varphi(0)$ from Definition~\ref{df:type1}.
\end{lem}

\begin{proof}
  We proceed by induction on the length of $w$. Suppose first that $|w|=1$. If $w=0$ then
  \[
  \Er(\su\varphi(0)) = \Er(\varphi(0))\Er(\su) = \Er(\pr\su)\su = \su\Er(\pr)\su = \su\varphi(1)
  = \su\varphi(\Er(0)).
  \]
  Otherwise $w=1$ and then
  \[
  \Er(\su\varphi(1)) = \Er(\varphi(1))\Er(\su) = \Er(\Er(\pr)\su)\su =
  \su\pr\su = \su\varphi(0) = \su\varphi(\Er(1)).
  \]
  Now let $|w|>1$, i.e., $w=w_1\cdots w_n$. We have
  \begin{align*}
    \Er(\su\varphi(w_1w_2\cdots w_n)) & = \Er(\varphi(w_2\cdots w_n))\Er(\su\varphi(w_1)) = \\
    & \stackrel{(1)}{=} \Er(\varphi(w_2\cdots w_n))\su\varphi(\Er(w_1)) = \\
    & \stackrel{(2)}{=} \Er(\su\varphi(w_2\cdots w_n))\varphi(\Er(w_1)) = \\
    & \stackrel{(3)}{=} \su\varphi(\Er(w_2\cdots w_n))\varphi(\Er(w_1)) = \su\varphi(\Er(w_1\cdots w_n)),
  \end{align*}
  where we have used (1) validity of the statement for $|w|=1$, (2) the fact that $\su$ is an
  antipalindrome, (3) induction hypothesis for $|w|=n-1$.
\end{proof}

\begin{proof}[Proof of Proposition~\ref{thm:type1_is_antipalindromic}]

  Let $w\in\Lc(\bu)$ be an antipalindrome. Lemma~\ref{lem:type1_and_E} implies that the word
  $\su\varphi(w)$ is also an antipalindrome. Indeed, $\Er(\su\varphi(w)) = \su\varphi(\Er(w)) =
  \su\varphi(w)$.

  Since $\varphi$ is a primitive morphism, its fixed point $\bu$ is uniformly recurrent, and thus
  for every $v\in\Lc(\bu)$ there exists $c\in\{0,1\}$ such that $cv\in\Lc(\bu)$. We then have
  $\varphi(cv)=\varphi(c)\varphi(v)\in\Lc(\bu)$. Moreover, since $\su\varphi(v)$ is a proper suffix
  of $\varphi(c)\varphi(v)$ we have $\su\varphi(v)\in\Lc(\bu)$.

  Therefore the image of an antipalindrome $w\in\Lc(\bu)$ under the mapping $w\mapsto\su\varphi(w)$
  is a longer antipalindrome in $\Lc(\bu)$.  By the assumption of primitivity of $\varphi$, we have
  $0,1\in\Lc(\bu)$ and thus either $01$ or $10$ is a factor of $\bu$. Therefore $\Lc(\bu)$
  contains a nonempty antipalindrome. The statement follows.
\end{proof}

We now give a necessary and sufficient condition for a morphism in class $\Aj$ to have a fixed point
with arbitrarily long palindromes.

\begin{prop}\label{thm:form_of_type1_with_palindromic_fixedp}
  Let $\varphi$ be a morphism in class $\Aj$ and let $\bu$ be its aperiodic fixed point. Then $\bu$
  is palindromic if and only if $\su=\varepsilon$ and $\pr$ is a palindrome.
\end{prop}

\begin{proof}
  First realize that if $\su=\varepsilon$ and $\pr$ is a palindrome then by Observation~\ref{obs},
  $\Er(\pr)$ is a palindrome and thus $\varphi$ belongs to class $\Pc$. Consequently, $\Lc(\bu)$
  contains infinitely many palindromes.

  For the opposite implication, let us find $\varphi_L$ and $\varphi_R$, i.e., the leftmost and
  rightmost conjugate respectively to $\varphi$. Let $x,y\in\{0,1\}^*$ be such that $\pr=xy\Er(x)$,
  where $x$ is the longest possible.  That is, either $y=\varepsilon$ or the first and the last
  letter of $y$ coincide.  Therefore $\varphi(0)=xy\Er(x)\su$ and $\varphi(1)=x\Er(y)\Er(x)\su$. If
  $y=\varepsilon$ then $\varphi(0)=\varphi(1)$ and a fixed point of $\varphi$ is periodic, a
  contradiction. Thus $y\neq\varepsilon$ and
  $$
  \begin{aligned}
  \varphi_L(0) &=y\Er(x)\su x,\\
  \varphi_R(0) &=\Er(x)\su xy,
  \end{aligned} \quad
  \begin{aligned}
  \varphi_L(1) &=\Er(y)\Er(x)\su x,\\
  \varphi_L(1)&=\Er(x)\su x\Er(y).
  \end{aligned}
  $$
  By Proposition~\ref{p:inverse_morphism}, $\Rr(\varphi_L(0))=\varphi_R(0)$, which implies
  that $y=\Rr(y)$, $\Rr(x)=\Er(x)$, and $\Rr(\su)=\su$. It follows that $\su=\varepsilon$,
  $x=\varepsilon$, and $y$ is a palindrome. Then indeed
  \begin{align*}
    \varphi_L(0) = \varphi_R(0) = \varphi(0) & = y, \\
    \varphi_L(1) = \varphi_R(1) = \varphi(1) & = \Er(y).
    \qedhere
  \end{align*}
\end{proof}

\subsection{Class $\Ad$ -- non-uniform morphisms}

The second class of considered morphisms contains morphisms that are non-uniform in general. For its
definition we use the Thue-Morse morphism~$\Theta$.

\begin{df}\label{df:type2}
  A morphism $\psi:\{0,1\}^*\rightarrow\{0,1\}^*$ is said to be in class $\Ad$ if there exist a
  non-empty word $\fa\in\{0,1\}^*$ and $k,h\in\N$ such that
  \[
  \psi(0) = \Theta\big(\fa(\Rr({\fa})\fa)^k\big), \qquad
  \psi(1) = \Theta\big((\Rr({\fa})\fa)^h\Rr({\fa})\big).
  \]
\end{df}

Note that a morphism in class $\Ad$ is necessarily primitive.

\begin{rem}\label{rem:1}
  If $\psi$ is in class $\Ad$ with $k=h$ then $\psi$ is also in class $\Aj$, where $\su=\varepsilon$
  and $\pr = \Theta\big(\fa(\Rr({\fa})\fa)^k\big)$. That is a consequence of the relation $\Theta R
  = E \Theta$ (cf.\ Observation~\ref{obs}), as
  $E\big(\psi(0)\big)=E\Theta\big(\fa(\Rr({\fa})\fa)^k\big)=\Theta \Rr
  \big(\fa(\Rr({\fa})\fa)^k\big)=\Theta\big((\Rr({\fa})\fa)^k\Rr({\fa})\big)=\psi(1)$.
\end{rem}

\begin{rem}\label{rem:EPA2}
  One can show that if $\psi$ is in class $\Ad$, then $\psi\Theta$ belongs to $\Ec$. Indeed,
  $(\psi\Theta)(0)=\psi(01)= \Theta\big((\fa\Rr({\fa}))^{k+h+1}\big)$ and $(\psi\Theta)(1)=\psi(10)=
  \Theta\big((\Rr({\fa})\fa)^{k+h+1}\big)$. Clearly, both $(\fa\Rr({\fa}))^{k+h+1}$ and
  $(\Rr({\fa})\fa)^{k+h+1}$ are palindromes. Using the fact that $\Theta(v)$ is an antipalindrome if
  $v$ is a palindrome, we derive that $\psi\Theta$ is of the desired form.  The opposite implication
  is not obvious.
\end{rem}

\begin{prop}\label{thm:type2_is_antipalindromic}
  Let $\bu$ be a fixed point of a morphism in class $\Ad$. Then $\Lc(\bu)$ contains infinitely many
  antipalindromes.
\end{prop}

We first prove two auxiliary lemmas.

\begin{lem}\label{lem:type2_lem1}
  Let $\psi$ be a morphism in class $\Ad$ and let $v\in\{0,1\}^*$. Then
  $\psi(\Theta(\Rr(v)))=\Er(\psi(\Theta(v)))$.
\end{lem}

\begin{proof}
  Since both $\psi\Theta\Rr$ and $\Er\psi\Theta$ are antimorphisms it is enough to prove the formula
  just for the letters $0,1$. We have
  \begin{align*}
      \psi(\Theta(0)) &= \psi(01) =
      \Theta\big(\underbrace{(\fa\Rr({\fa}))^{k+h+1}}_{\text{palindrome}}) =
      \Theta\big(\Rr(\fa\Rr({\fa}))^{k+h+1}\big) = \\ &= \Er\Theta\big((\fa\Rr({\fa}))^{k+h+1}\big)
      = \Er(\psi(01)) = \Er(\psi(\Theta(0))), \intertext{and} \psi(\Theta(1)) &= \psi(10) =
      \Theta\big(\underbrace{(\Rr({\fa})\fa)^{k+h+1}}_{\text{palindrome}}) =
      \Theta\big(\Rr(\Rr({\fa})\fa)^{k+h+1}\big) = \\ &= \Er\Theta\big((\Rr({\fa})\fa)^{k+h+1}\big)
      = \Er(\psi(10)) = \Er(\psi(\Theta(1))),
  \end{align*}
  where we used the fact that $\Theta\Rr=\Er\Theta$.
\end{proof}

\begin{lem}\label{lem:type2_lem2}
  Let $\psi$ be a morphism in class $\Ad$ and let $v\in\{0,1\}^*$ be such that $\Theta(v)$ is an
  antipalindrome. Then $\psi(\Theta(v))$ is an antipalindrome.
\end{lem}

\begin{proof}
  Recall from Observation~\ref{obs} that $\Theta(v)$ is an antipalindrome if and only if $v$ is a palindrome.
  Then
  $\Er(\psi(\Theta(v)))=\Er(\psi(\Theta(\Rr(v))))=\psi(\Theta(\Rr(\Rr(v)))) =
  \psi(\Theta(v))$, where the pen\-ultimate equality follows from Lemma~\ref{lem:type2_lem1}.
\end{proof}

\begin{proof}[Proof of Proposition~\ref{thm:type2_is_antipalindromic}]
  Obviously, the fixed point $\bu$ of $\psi$ contains either the antipalindrome $01$ or $10$. By
  repeated application of Lemma~\ref{lem:type2_lem2} on it, we get increasingly longer
  antipalindromes in $\Lc(\bu)$.
\end{proof}

As in the case of morphisms in class $\Aj$, we give a necessary and sufficient condition for a
morphism in class $\Ad$ to have a fixed point with arbitrarily long palindromes.

\begin{prop}\label{thm:form_of_type2_with_palindromic_fixedp}
  Let $\psi$ be a morphism in class $\Ad$ and let $\bu$ be its aperiodic fixed point. Then $\bu$ is
  palindromic if and only if $\fa$ from Definition~\ref{df:type2} is an antipalindrome.
\end{prop}

\begin{proof}
  If $\fa$ is an antipalindrome, then by Observation~\ref{obs}, $\Theta(\fa)$ is a palindrome, and
  consequently, also $\psi(0)$ and $\psi(1)$ are palindromes. Therefore the morphism $\psi$ belongs
  to class $\Pc$, and the language $\Lc(\bu)$ contains infinitely many palindromes.

  Suppose on the other hand that $\bu$ is palindromic.  Let $\fa=xy\Rr(x)$, for some
  $x,y\in\{0,1\}^*$ and let $x$ be the longest prefix permitting to write $\fa$ in such
  form. Necessarily, $y\neq\varepsilon$, since otherwise $\psi(0)=\psi(1)$ and $\bu$ is
  periodic. Moreover, $\mathrm{fst}(y)\neq\mathrm{lst}(y)$.  Therefore for the leftmost and
  rightmost conjugate to $\psi$ we have
  \begin{align*}
    \psi_L (0) &=
    \Theta\big(y\Rr(x)(x\Rr(y)\Rr(x)xy\Rr(x))^k\big)\Theta(x),\\
    \psi_L (1) &=
    \Theta\big(\Rr(y)\Rr(x)(xy\Rr(x)x\Rr(y)\Rr(x))^h\big)
    \Theta(x),\\
    \psi_R (0) &=
    \Theta(\Rr(x))\Theta\big((xy\Rr(x)x\Rr(y)\Rr(x))^kxy\big),\\
    \psi_R (1) &=
    \Theta(\Rr(x))
    \Theta\big((x\Rr(y)\Rr(x)xy\Rr(x))^hx\Rr(y)\big).
  \end{align*}
  Since $\bu$ is palindromic, by Proposition~\ref{p:inverse_morphism}, we have
  $\Rr\big(\psi_L(a)\big)=\psi_R(a)$ for $a=0,1$. Comparing with the above we obtain that
  $\Rr\Theta(x)=\Theta\Rr(x)=\Er\Theta(x)$ and thus $x=\varepsilon$. Further, we have
  $\Rr\Theta(y)=\Theta(y)$, which means that $\Theta(y)$ is a palindrome and thus (again by
  Observation~\ref{obs}) $y=\fa$ is an antipalindrome.
\end{proof}

The following proposition shows that changing the parameters $k,h$ in the definition of a morphism
$\psi$ in class $\Ad$ while keeping their sum $k+h$ fixed does not change the fixed points.

\begin{prop}\label{p:notrigid}
  Let $\psi$ be a morphism in class $\Ad$ with $k+h\geq 1$. Let $\xi:\{0,1\}^*\rightarrow\{0,1\}^*$ be
  given by
  \[
  \xi(0) = \Theta(\fa) \qquad\text{and}\qquad
  \xi(1) = \Theta\big((\Rr(\fa)\fa)^{k+h}\Rr(\fa)\big).
  \]
  Then $\bu$ is a fixed point of $\psi$ if and only if it is a fixed point of $\xi$.
\end{prop}

\begin{proof}
  Note that
  \begin{align*}
    \xi(\Theta(0)) &= \xi(01) = \psi(01) = \psi(\Theta(0)) \\
    \xi(\Theta(1)) &= \xi(10) = \psi(10) = \psi(\Theta(1)),
  \end{align*}
  and thus
  \begin{equation}\label{eq:psi_and_xi}
    \xi\Theta = \psi\Theta.
  \end{equation}
  Since $\bu=\psi(\bu)$ there is an infinite word $\boldsymbol{v}$ such that $\bu =
  \Theta(\boldsymbol{v})$.  Using~(\ref{eq:psi_and_xi}) we get
  \[
  \xi(\bu) = \xi(\Theta(\boldsymbol{v})) = \psi(\Theta(\boldsymbol{v})) = \psi(\bu) = \bu.
  \]
  The argumentation for the opposite implication is the same, interchanging the role of $\xi $ and
  $\psi$.
\end{proof}

\begin{rem}\label{rem:2}
  Proposition~\ref{p:notrigid} implies that a morphism $\xi$ in class $\Ad$ with $k+h\geq 1$ is not
  rigid. On the other hand, if $\psi$ is in class $\Ad$ with $k+h=0$ and $\fa=01$ then
  $\psi=\Theta^2$ is an iterate of the Thue-Morse morphism $\Theta$. It is shown
  in~\cite{pansiot-ipl-12} that its fixed points are rigid. For $k+h=0$ and $\fa\neq 01$, the
  situation is not clear.
\end{rem}

\section{Palindromes and antipalindromes in periodic words}\label{sec:periodic}

Let us first study properties of eventually periodic words which contain arbitrarily long
antipalindromes. The statements we show here are direct analogues of similar facts for palindromes,
namely Propositions~\ref{lem:period_in_evper_pal} and~\ref{lem:classP_evper_pal} which we cite
from~\cite{brlek-hamel-nivat-reutenauer-ijfcs-15,allouche-baake-cassaigne-damanik-tcs-292}.

\begin{prop}[\cite{brlek-hamel-nivat-reutenauer-ijfcs-15}]\label{lem:period_in_evper_pal}
  Let $\bu=vw^{\infty}$, where $w\neq\varepsilon$ be an eventually periodic palindromic word.  Then
  $w=pq$ where $p$, $q$ are palindromes.
\end{prop}

\begin{prop}[\cite{allouche-baake-cassaigne-damanik-tcs-292}]\label{lem:classP_evper_pal}
  Let $\bu$ be an eventually periodic palindromic word.  If $\bu$ is recurrent then $\bu$ is a fixed
  point of a morphism in class $\Pc$.
\end{prop}

We prove the analogues of the above two results for antipalindromic words as
Lemmas~\ref{lem:period_in_evper_antipal} and~\ref{lem:evper_antipal_fixed_by_Type1}.
Lemma~\ref{lem:period_in_evper_antipal} has been also proven in~\cite{labbe-memoire}.

\begin{lem}\label{lem:period_in_evper_antipal}
  Let $\bu=vw^{\infty}$, where $w\neq\varepsilon$, be an eventually periodic antipalindromic
  word. Then $w=w_1w_2$ where $w_1,w_2$ are antipalindromes.
\end{lem}

\begin{proof}
  Since $\Lc(\bu)$ contains infinitely many antipalindromes it necessarily contains antipalindromes
  in the form $bw^kc$, where $k\geq 1$, and $b$ and $c$ is a proper suffix and a proper prefix of
  $w$, respectively. Without loss of generality we can assume that $|b|\leq|c|<|w|$. Let $w_1$ be
  the prefix of $c$ of length $|c|-|b|$. Then $w^kw_1$ is also an antipalindrome. Let us denote
  $w=w_1w_2$. We have
  \[
  (w_1w_2)^kw_1 = \Er\big((w_1w_2)^kw_1\big) = \Er(w_1)\big(\Er(w_2)\Er(w_1)\big)^k.
  \]
  Thus $w_1=\Er(w_1)$ and $w_2=\Er(w_2)$.
\end{proof}

\begin{lem}\label{lem:evper_antipal_fixed_by_Type1}
  Let $\bu$ be an eventually periodic antipalindromic word.  If $\bu$ is recurrent then it is a
  fixed point of a morphism in class $\Aj$.
\end{lem}

\begin{proof}
  A recurrent eventually periodic word is necessarily periodic, i.e., $\bu=w^{\infty}$, where
  $w=w_1w_2$ by Lemma~\ref{lem:period_in_evper_antipal}. The word $\bu$ is obviously fixed by the
  morphism given by $0\mapsto w$ and $1\mapsto w$, which is in class $\Aj$ ($\pr=w_1$, $\su=w_2$).
\end{proof}

The following theorem puts together the above facts and gives a description of eventually periodic
words which are both palindromic and antipalindromic.

\begin{prop}
  Let $\bu$ be an eventually periodic word which contains infinitely many palindromes and
  antipalindromes. Then there exist a word $b\in\{0,1\}^*$ and a palindrome $c\in\{0,1\}^*$ such
  that $\bu=b\big(c\Er(c)\big)^{\infty}$.
\end{prop}

\begin{proof}
  If $\bu=vw^{\infty}$, where $w\neq\varepsilon$, then by Lemma~\ref{lem:period_in_evper_antipal} we
  have $w=w'w''$, where $w'$, $w''$ are antipalindromes. Thus $w'=\Er(f)f$ for some $f$ and
  $\bu=v\Er(f)\big(fw''\Er(f)\big)^{\infty}$. Therefore $\bu$ is of the form $\bu=ba^\infty$ where
  $a=fw''\Er(f)$ is an antipalindrome.  By Proposition~\ref{lem:period_in_evper_pal} $a$ can be
  written as a concatenation of two palindromes $a=pq$. Using item (iii) of
  Proposition~\ref{p:properties} we infer that there is a palindrome $c$ such that
  $a=\big(c\Er(c)\big)^k$ for some $k\in\N$.
\end{proof}


\section{Languages of aperiodic words containing infinitely many antipalindromes}\label{sec:language}

In this section we will demonstrate several properties of aperiodic antipalindromic words which will
be used in the proof of our main results.
Lemmas~\ref{lem:unifrec_closed_under_E},~\ref{lem:inf_word_of_antipals}
and~\ref{lem:bispecials_in_antipalin_word} are analogues of similar statements for palindromic
words.


\begin{lem}\label{lem:unifrec_closed_under_E}
  Let $\bu$ be a uniformly recurrent antipalindromic word. Then $\Lc(\bu)$ is closed under the
  antimorphism $\Er$.
\end{lem}

\begin{proof}
  Since $\bu$ is uniformly recurrent there is a function $r:\N\to\N$ such that every factor of $\bu$
  of length $r(n)$ contains all the factors of $\bu$ of length $n$.  Let $w\in\Lc(\bu)$ and let
  $v\in\Lc(w)$ be an antipalindrome such that $|v|\geq r(|w|)$.  Then both $w$ and $\Er(w)$ are
  factors of $\bu$.
\end{proof}


\begin{lem}\label{lem:inf_word_of_antipals}
  Let $\bu$ be an infinite antipalindromic word. Then there is an infinite word $\boldsymbol{v}$
  such that $\Er(v)v\in\Lc(\bu)$ for every prefix $v$ of $\boldsymbol{v}$.
\end{lem}
\begin{proof}
  Every antipalindrome is of the form $\Er(w)w$. Since $\bu$ contains infinitely many
  antipalindromes, there are an infinite number of words $w$ such that $\Er(w)w\in\Lc(\bu)$.  Define
  $\boldsymbol{v}=v_1v_2v_3\cdots$ in the following way: let $v_1$ be a letter which is a prefix of
  infinitely many $w$ such that $\Er(w)w\in\Lc(\bu)$; let $v_2$ be a letter for which $v_1v_2$ is a
  prefix of infinitely many $w$ such that $\Er(w)w\in\Lc(\bu)$; etc.
\end{proof}

\begin{rem}
  If $\bu$ is uniformly recurrent then the languages of $\bu$ and of the (both-sided) infinite word
  $\Er(\boldsymbol{v})\boldsymbol{v}$ coincide. Otherwise, we have
  $\Lc(\Er(\boldsymbol{v})\boldsymbol{v})\subset\Lc(\bu)$.
\end{rem}

The following lemma is an analogy of Lemma 25 of~\cite{labbe-pelantova-ejc-51} where the authors
study words containing infinitely many palindromes.

\begin{lem}\label{lem:bispecials_in_antipalin_word}
  Let $\bu$ be an aperiodic uniformly recurrent antipalindromic word. Then $\Lc(\bu)$ contains
  infinitely many antipalindromic bispecial factors.
\end{lem}

\begin{proof}
  We use the fact that every factor $f$ of an aperiodic word $\bu$ which is not itself right special
  has unique minimal (right) prolongation into a right special factor of $\bu$, i.e., there is a
  unique word $e$ such that $fe$ is a right special in $\Lc(\bu)$ while $fe'$ is not right special
  for each proper prefix $e'$ of $e$. An analogous claim holds also for left special factors.

  For each $N\in\N$ we demonstrate how to construct an antipalindromic bispecial factor in
  $\Lc(\bu)$ of length greater than or equal to $N$. Let $w$ be a prefix of the infinite word
  $\boldsymbol{v}$ from Lemma~\ref{lem:inf_word_of_antipals} (recall that $\Er(w)w\in\Lc(\bu)$) such
  that $|\Er(w)w|\geq N$.  If $\Er(w)w$ is right special then it is left special too, since by
  Lemma~\ref{lem:unifrec_closed_under_E} the language $\Lc(\bu)$ is closed under $\Er$. Let
  $\Er(w)w$ be not a right special factor. Then there exists $e$ such that $\Er(w)we$ is right
  special. Recall that for every proper prefix $e'$ of $e$ the factor $\Er(w)we'$ has unique right
  extension, thus $we$ is also a prefix of $\boldsymbol{v}$.

  Since $\Lc(\bu)$ is closed under $\Er$, $\Er(e)\Er(w)w$ is unique left prolongation of $\Er(w)w$,
  and, moreover, it is a left special factor. Therefore the word $\Er(e)\Er(w)we$ is a bispecial
  factor in $\Lc(\bu)$.
\end{proof}

The last lemma of this section speaks about frequencies of letters in an antipalindromic word. This
result has no counterpart in palindromic words.  Here we assume that such frequencies exist, which
is not necessarily true even for uniformly recurrent words. Later on, this assumption will be
ensured by taking fixed points of primitive morphisms.

\begin{lem}\label{lem:freq_of_letters_in_binary_antipalindromic}
  Let $\bu$ be a binary infinite word containing infinitely many
  antipalindromes. Suppose that the frequency of a letter $a$ in $\bu$
  exists, namely
  \[
  \varrho(a) = \lim_{|w|\to\infty}\frac{|w|_a}{|w|}.
  \]
  Then the frequency of both letters in $\bu$ is equal to $1/2$.
\end{lem}

\begin{proof}
  Since the frequency of a letter $a$ in $\bu$ exists, it can be computed as corresponding limit of
  any subsequence of factors of $\bu$.  Let us consider a subsequence obtained by taking only
  antipalindromic factors $w$ of $\bu$. As $\Er(w)=w$, we have
  \[
  \frac{|w|_a}{|w|} = \frac{1}{2} \qquad \text{for $a=0,1$.}
  \]
  Therefore $\varrho(0)=\frac12=\varrho(1)$.
\end{proof}


\section{Antipalindromes in fixed points of morphisms}\label{sec:main}

In this section we present our main results that concern antipalindromic fixed points of primitive
morphisms.  The case of periodic words is covered by Lemma~\ref{lem:evper_antipal_fixed_by_Type1},
thus we focus on aperiodic fixed points.

The languages of infinite words fixed by mutually conjugated primitive morphisms coincide, cf.\ item
(iii) of Proposition~\ref{p:properties}. For a chosen morphism $\varphi$ with an aperiodic
antipalindromic fixed point we consider the rightmost and leftmost conjugate $\varphi_R$,
$\varphi_L$ and the relation between them, namely the conjugacy word $q\in\{0,1\}^*$ such that
\begin{equation}\label{eq:dovaleny_morfizmus}
  q\varphi_L(w) = \varphi_R(w)q \qquad \text{for every $w\in\{0,1\}^*$.}
\end{equation}

Crucial for the demonstration of our results is the description of antipalindromic bispecial factors
in fixed points of primitive morphisms derived in the sense of Theorem~\ref{thm:klouda}.  The
following lemma testifies about them and the conjugacy word $q$ of~\eqref{eq:dovaleny_morfizmus}.

\begin{lem}\label{lem:q_is_antipalindrome}
  Let $\varphi$ be a primitive binary morphism with an aperiodic antipalindromic fixed point $\bu$.
  Then the word $q$ from~(\ref{eq:dovaleny_morfizmus}) is an antipalindrome and $\Lc(\bu)$ contains
  infinitely many antipalindromes of the form $\varphi_R(w)q$ for some $w\in\Lc(\bu)$.

  If, moreover, $\bu$ contains infinitely many palindromes, then $q=\varepsilon$ and $\Lc(\bu)$
  contains infinitely many antipalindromes of the form $\varphi(w)$ for some $w\in\Lc(\bu)$.
\end{lem}

\begin{proof}
  By Lemma~\ref{lem:bispecials_in_antipalin_word} the language $\Lc(\bu)$ contains infinitely many
  antipalindromic bispecial factors. By Theorem~\ref{thm:klouda}, any sufficiently long bispecial
  factor is of the form $\varphi_R(w)q$, where $w\in\Lc(\bu)$. Consider $w$ such that
  $\varphi_R(w)q$ is an antipalindrome.  Using~(\ref{eq:dovaleny_morfizmus}) we obtain
  \begin{equation}\label{eq:bispecials_w}
  q\varphi_L(w)=\varphi_R(w)q = \Er\big(\varphi_R(w)q\big) = \Er(q)\Er\big(\varphi_R(w)\big).
  \end{equation}
  Therefore $q=\Er(q)$.
  
  If $\bu$ contains also infinitely many palindromes, by Lemma 20
  from~\cite{labbe-pelantova-ejc-51}, $q$ is a palindrome. Altogether, $q=\varepsilon$ and
  consequently $\varphi_R=\varphi$. Thus $\varphi_R(w)q=\varphi(w)$.
\end{proof}

We are now in position to state our main results. They are separated into two theorems. The first
one of them, Theorem~\ref{thm:mainUnifrom}, states for uniform morphisms the opposite of
Proposition~\ref{thm:type1_is_antipalindromic}.

\begin{thm}\label{thm:mainUnifrom}
  Let $\bu$ be an aperiodic fixed point of a primitive binary uniform morphism $\varphi$ such that
  $\bu$ contains infinitely many antipalindromes. Then $\varphi$ or $\varphi^2$ is conjugated to a
  morphism in class $\Aj$.
\end{thm}

\begin{proof}
  Since $\varphi$ has an aperiodic fixed point, it is acyclic. Therefore the leftmost and rightmost
  conjugates $\varphi_L$, $\varphi_R$ to the morphism $\varphi$ exist and $\varphi$ is
  marked. Without loss of generality we can assume for the first and last letters that
  \begin{equation}\label{eq:borders_of_varphiL_varphiR}
    \begin{aligned}
      \mathrm{fst}(\varphi_L(a)) &= a & & \quad \text{for $a=0,1$}, \\
      \mathrm{lst}(\varphi_R(a)) &= a & & \quad  \text{for $a=0,1$}. \\
    \end{aligned}
  \end{equation}
  Otherwise, we take $\varphi^2$ instead of $\varphi$.

  Let $w=w_1w_2\cdots w_n$ be a factor such that $\varphi_R(w)q$ is an antipalindrome, which exists
  by Lemma~\ref{lem:q_is_antipalindrome}.  Then, similarly as in~(\ref{eq:bispecials_w}) we have
  \[
  \varphi_L(w_1)\varphi_L(w_2)\cdots\varphi_L(w_n) =
  \Er\big(\varphi_R(w_n)\big)\Er\big(\varphi_R(w_{n-1})\big)\cdots\Er\big(\varphi_R(w_1)\big).
  \]
  From the equation~(\ref{eq:borders_of_varphiL_varphiR}) and from the property
  $|\varphi(0)|=|\varphi(1)|$ we see that
  \[
  w_1 = \Er(w_n), w_2 = \Er(w_{n-1}), \ldots
  \]
  and, consequently, for every letter $a$ we have
  \begin{equation}\label{eq:varphiL_varphiR_relation_on_letters}
    \varphi_L(\Er(a)) = \Er\big(\varphi_R(a)\big).
  \end{equation}

  There are two different cases based on the form of the conjugacy word $q$ from
  ~(\ref{eq:dovaleny_morfizmus}).  If $q$ is empty then $\varphi_L = \varphi_R = \varphi$,
  by~(\ref{eq:varphiL_varphiR_relation_on_letters}) it fulfills
  $\varphi(1)=\Er\big(\varphi(0)\big)$, and, therefore, $\varphi$ is in class $\Aj$.

  Let $q\neq\varepsilon$. Combining~(\ref{eq:dovaleny_morfizmus})
  and~(\ref{eq:varphiL_varphiR_relation_on_letters}) we get
  \begin{equation}\label{eq:varphiL_varphiR_equation}
    \varphi_R(1)q = q\varphi_L(1) = q\Er\big(\varphi_R(0)\big).
  \end{equation}
  Let us denote $x := \varphi_R(1)$ and $y := \Er\big(\varphi_R(0)\big)$. Then the
  formula~(\ref{eq:varphiL_varphiR_equation}) can be seen as an equation on words of the form
  $xq=qy$.  Thus by item (ii) of Proposition~\ref{p:properties}, there exist $e,f\in\{0,1\}^*$ and
  $i\in\N$ such that
  \begin{equation}\label{eq:varphiL_varphiR_equation_solution}
    x = ef, \quad y = fe, \quad q = (ef)^ie.
  \end{equation}
  At first, let us inspect the case where $i=0$. Since by Lemma~\ref{lem:q_is_antipalindrome} the
  word $q=e$ is an antipalindrome, we can write
  \[
    \varphi_R(1) = ef \quad\text{and}\quad \varphi_R(0) = \Er(y) = \Er(fe) = e\Er(f),
  \]
  and, therefore, $\varphi$ is conjugated to a morphism in class $\Aj$.

  Now, let $i\geq1$. Since $q=(ef)^ie$ is an antipalindrome, we have that $e$, $f$ are
  antipalindromes, too. This implies that $\varphi_R(1)=ef=\Er(fe)=\varphi_R(0)$, i.e., the fixed
  point is periodic. This is a contradiction, and, therefore, the case $i\geq 1$ will not occur.
\end{proof}

As an application of Theorem~\ref{thm:mainUnifrom} we will show an example of a uniform morphism in
class $\Ec$ which does not generate a palindromic fixed point.

\begin{ex}\label{ex:EP}
  Consider the binary morphism $\varphi$ defined by $\varphi(0)=0101$ and $\varphi(1)=1100$. Since
  both $\varphi(0)$ and $\varphi(1)$ are antipalindromes, $\varphi$ belongs to class $\Ec$.  At the
  same time, one can check that neither $\varphi$ nor $\varphi^2$ belongs to $\Aj$. Therefore by
  Theorem~\ref{thm:mainUnifrom}, the fixed point of $\varphi$ contains only finitely many
  antipalindromes.
\end{ex}


In the second main theorem, we show that only morphisms of class $\Ad$ have aperiodic fixed points
containing infinitely many of both antipalindromes and palindromes.  Even with the additional
requirement of palindromicity of $\bu$, the proof of Theorem~\ref{thm:mainNonuniform} is
complicated. It is likely that the demonstration of the fact that any non-uniform morphism
generating aperiodic infinite antipalindromic words is in class $\Ad$ would require a novel approach
using new techniques.

\begin{thm}\label{thm:mainNonuniform}
  Let $\bu$ be an aperiodic fixed point of a primitive binary non-uniform morphism $\varphi$ such
  that $\Lc(\bu)$ contains an infinite number of palindromes as well as antipalindromes. Then either
  $\varphi$ or $\varphi^2$ is a morphism in class $\Ad$ with $\fa$ being an antipalindrome.
\end{thm}

For the proof of Theorem~\ref{thm:mainNonuniform} we need, besides the properties of languages of
antipalindromic words demonstrated in Section~\ref{sec:language}, also a number of auxiliary
statements. We collect them together with the proof of Theorem~\ref{thm:mainNonuniform} in the
following section.


\section{Proof of Theorem~\ref{thm:mainNonuniform}}\label{sec:proof}

Throughout this section we will work only with morphisms of a special
form. Lemma~\ref{lem:assumption} then shows that such a restriction is justified by assumptions of
Theorem~\ref{thm:mainNonuniform}.

\smallskip
\noindent\textbf{Condition $(\star)$}: We say that $\varphi$ and $\bu$ satisfy condition $(\star)$
if $\varphi$ is a primitive binary non-uniform morphism such that $\varphi(a)$ is a palindrome with
prefix $a$ for $a=0,1$, and that $\bu$ is its aperiodic antipalindromic fixed point.

\begin{lem}\label{lem:assumption}
  Let $\bu$ be an aperiodic fixed point of a primitive binary non-uniform morphism $\varphi$ such
  that $\Lc(\bu)$ contains an infinite number of palindromes as well as antipalindromes. Then $\bu$
  with either $\varphi$ or $\varphi^2$ satisfy condition $(\star)$.
\end{lem}

\begin{proof}
  Since $\bu$ is palindromic, it follows from a result by Tan~\cite{tan-tcs-389} that $\varphi$ or
  $\varphi^2$ is a conjugate to $\xi$ from class $\Pc$, i.e., $\xi(0)=p_0p$ and $\xi(1)=p_1p$, where
  $p_0,p_1,p$ are palindromes. Since $\xi$ is also a primitive morphism, the language of its fixed
  point coincides with $\Lc(\bu)$. By Lemma~\ref{lem:q_is_antipalindrome}, the conjugacy word $q$
  between morphism $\xi_R,\xi_L$ is empty, and thus we have $\xi=\xi_R=\xi_L$. Therefore there is no
  other morphism conjugated with $\xi$, and thus either $\varphi$ or $\varphi^2$ is equal to $\xi$,
  and, moreover, the palindrome $p$ is empty. This means that either $\xi=\varphi$ or
  $\xi=\varphi^2$ together with $\bu$ satisfy condition $(\star)$.
\end{proof}

The above lemma shows that under the assumptions of Theorem~\ref{thm:mainNonuniform}, the words
$\varphi(0)$, $\varphi(1)$ are palindromes. In order to complete the proof of
Theorem~\ref{thm:mainNonuniform}, we need to show that these palindromes are of a very special form
given by Definition~\ref{df:type2}. The demonstration needs two technical statements, which follow
as Lemmas~\ref{lem:tvar} and~\ref{lem:ttvar}. For the sake of lucidity, we have separated parts of
the proof of Lemma~\ref{lem:tvar} into two auxiliary facts, formulated as Lemmas~\ref{lem:vyluc}
and~\ref{lem:primit}.

\begin{lem}\label{lem:vyluc}
  Let $\varphi$ and $\bu$ satisfy condition $(\star)$.  Suppose that there exists a nonempty word
  $x\in\{0,1\}^*$ such that $x\neq \Er(x)$ and $\varphi(0),\varphi(1)\in\{x,\Er(x)\}^*$. For $a=0,1$
  denote by $k_a$, $h_a$ the number of words $x$, $\Er(x)$, respectively, in $\varphi(a)$. Then
  $k_0+k_1= h_0+h_1$.
\end{lem}

\begin{proof}
  We shall prove the statement by contradiction. Suppose that $\Lc(\bu)$ contains infinitely many
  antipalindromes and $k_0+k_1\neq h_0+h_1$. Then at least one of $k_0-h_0$, $k_1-h_1$ is
  non-zero. Assume without loss of generality that $k_0\neq h_0$. Then $\frac{h_1-k_1}{k_0-h_0}\neq
  1$ and there exists $\delta>0$ such that $\frac{h_1-k_1}{k_0-h_0}\notin (1-\delta,1+\delta)$. As
  $\bu$ is a fixed point of a primitive morphism, the frequencies of letters
  exist~\cite{queffelec-lncs-1294}. By Lemma~\ref{lem:freq_of_letters_in_binary_antipalindromic} the
  frequencies of both letters in $\bu$ are equal to $\tfrac{1}{2}$, thus there is $N\in\N$ such that
  for every factor $w$ of $\bu$ of the length $|w|>N$ we have $\frac{|w|_0}{|w|_1}\in
  (1-\delta,1+\delta)$. Choose a factor $w$ such that $\varphi(w)$ is an antipalindrome in $\bu$ of
  length greater than $N(|\varphi(0)|+|\varphi(1)|)$. Such $w\in\Lc(\bu)$ exists by
  Lemma~\ref{lem:q_is_antipalindrome}. Evidently, $|w|>N$.  By assumption, $\varphi(w)$ is a
  concatenation of words $x$ and $\Er(x)$, where $|x|=|\Er(x)|$ and $x\neq\Er(x)$. Therefore
  $\varphi(x)$ has to consist of the same number of factors $x$ and $\Er(x)$, that is, $k_0|w|_0 +
  k_1|w|_1 = h_0|w|_0 + h_1|w|_1$. Consequently,
  \[
  \frac{|w|_0}{|w|_1} = \frac{h_1-k_1}{k_0-h_0} \notin (1-\delta,1+\delta).\qedhere
  \]
\end{proof}

Note that in the proof of Lemma~\ref{lem:vyluc} we have not used that $\varphi(0),\varphi(1)$ are
palindromes.

\begin{lem}\label{lem:primit}
  Let $c$ be a nonempty palindrome, and let $N\geq 1$. Then $\big(\Er(c)c\big)^N\Er(c)$ is a
  primitive word.
\end{lem}

\begin{proof}
  We shall prove by contradiction. Let $n\in\N$, $n\geq 2$ be the largest integer such that
  $z^n=\big(\Er(c)c\big)^N\Er(c)$ for some $z\in\{0,1\}^*$, and let $w:=z^n$. As $\Er(c)$ is a
  palindrome, both $z$ and $w$ are palindromes too. Obviously, the word $w$ has periods
  $|\Er(c)c|=2|c|$ and $|z|$.
  \begin{enumerate}
  \item
    At first, assume that $n+N\geq 4$. Thus either $N\geq 2$, or $N=1$ and $n\geq 3$.  By simple
    inspection, we derive from $|w|=(2N+1)|c|=n|z|$ that $|w|\geq 2|c|+|z|$. By
    Theorem~\ref{thm:fine-wilf}, $w$ has also the period $\ell=\gcd(|z|,2|c|)$. If $\ell<|z|$ then
    $|z|=\ell j$ for some $j\in\N$, $j\geq 2$ and therefore $z=y^j$ for some $y\in\{0,1\}^*$, a
    contradiction with the maximality of $n$. This implies $\gcd(|z|,2|c|)=|z|$, thus the period
    $|z|$ of $w$ divides $2|c|$, and there exists $s\in\N$ such that $z^s=\Er(c)c$. As $z$ is a
    palindrome it follows that $z^s=\Er(c)c$ is a palindrome, too, that is,
    \begin{equation}\label{eq:Ecc_pal}
      \Er(c)c = \Rr\big(\Er(c)c\big) = c\Er(c),
    \end{equation}
    where we used the fact that $c$ is a palindrome. By~(\ref{eq:Ecc_pal}) we have $c=\Er(c)$. The
    only word which is a palindrome and an antipalindrome at the same time is $c=\varepsilon$, a
    contradiction.
  \item
    Now, assume that $n+N=3$, i.e., $n=2$, $N=1$, and
    \begin{equation}\label{eq:zz}
      zz = \Er(c)c\Er(c).
    \end{equation}
    This means that the length of $c$ is even and we can write $c=ef$, where $|e|=|f|=\tfrac{1}{2}|c|$.
    Substituting $c=ef$ in~(\ref{eq:zz}) we have
    \[
    zz = \Er(ef)ef\Er(ef) = \Er(f)\Er(e)ef\Er(f)\Er(e),
    \]
    and since the first and the second half of the word have to be equal we obtain $\Er(f)=f$,
    $\Er(e)=\Er(f)$ and $e=\Er(e)$. Therefore $z=fff$, a contradiction with the maximality of $n$.
    \qedhere
  \end{enumerate}
\end{proof}

\begin{lem}\label{lem:tvar}
 Let $\varphi$ and $\bu$ satisfy condition $(\star)$.
 Then
 \begin{enumerate}
 \item
   there exists a palindrome $c\in\{0,1\}^*$ containing $00$ or $11$ such that
   \begin{equation}\label{eq:req_form_p0p1}
     \varphi(0)=\big(\Er(c)c\big)^{l_0}\Er(c)
     \quad\text{and}\quad
     \varphi(1)=\big(c\Er(c)\big)^{l_1}c,
   \end{equation}
   for some $l_0,l_1\in\N$;
 \item
   if $w\in\Lc(\bu)$ is such that $\varphi(w)$ is an antipalindrome, then
   there exists a palindrome $u$ such that $w=\Theta(u)$;
 \item
   every $v\in\Lc(\bu)$ can be expressed in the form $v=x\Theta(z)y$, where $x,y\in\{0,1,\varepsilon\}$,
   $z\in\{0,1\}^*$. Moreover, if either $00$ or $11$ is a factor of $v$ then $x,y,z$ are uniquely
   determined.
 \end{enumerate}
\end{lem}

\begin{proof}
  We start with the proof of the statement (i). Denote $p_0=\varphi(0)$, $p_1=\varphi(1)$, without
  loss of generality we can assume $|p_0|<|p_1|$. Consider a factor $w\in\Lc(\bu)$ such that
  $\varphi(w)$ is an antipalindrome, which exists by Lemma~\ref{lem:q_is_antipalindrome}.

  \smallskip

  \noindent
  At first, let us inspect the case where 0 is a prefix of $\varphi(w)$. As $\varphi(w)$ is an
  antipalindrome, it necessarily has a suffix 1. It follows from the form of $\varphi$ that $w$ has
  a prefix 0 and a suffix 1. Let $k\in\N$ be a positive integer such that $w$ has a prefix $0^k1$.

  \smallskip

  \noindent
  Case 1) $\exists i\in\N$, $1\leq i\leq k$ such that $|p_0^i|\geq|\Er(p_1)|>|p_0^{i-1}|$. Since
  $|p_0|<|p_1|$ we have $i\geq 2$. Recall that $\Er(\varphi(w))=\varphi(w)$, the situation is as
  follows
  \begin{align*}
    \varphi(w) &=
    \overbrace{\rule{0pt}{1.9ex}\boxed{\kern1ex{p_0}\kern1ex}\,\boxed{\kern1ex{p_0}\kern1ex}
      \,\cdots\,\boxed{\kern1ex{p_0}\kern1ex}}^{i\text{-times}}\boxed{\kern1ex{p_0}\kern1ex}\,\cdots\\
    \Er(\varphi(w)) &= \boxed{\kern6ex{\Er(p_1)}\kern6ex}\,\cdots
  \end{align*}
  Thus there exist $p',p''$ such that $p'\neq\varepsilon$, $p_0=p'p''$, and
  $\Er(p_1)=(p'p'')^{i-1}p'$. Since $i\geq 2$ and $\Er(p_1)$ is a palindrome, we have that $p'$ and
  $p''$ are palindromes. It follows that $p_0=p'p''$ is a palindrome too, and therefore $p'$ and
  $p''$ commute. Indeed, $p'p''=p_0=\Rr(p_0)=p''p'$. By item (ii) of Proposition~\ref{p:properties},
  this means that there is a palindrome $x$ such that $p'=x^t$ and $p''=x^s$ for
  some $t,s\in\N$, $t\geq 1$. Consequently, $\varphi$ is of the form $\varphi(0)=x^{t+s}$ and
  $\Er\big(\varphi(1)\big)=x^{(t+s)(i-1)+t}$. Since $\Lc(\bu)$ contains infinitely many palindromes,
  by Lemma~\ref{lem:vyluc}, we have $t+s=(t+s)(i-1)+t$, i.e., $0=(t+s)(i-2)+t\geq t\geq 1$, a
  contradiction. The Case 1) cannot occur.

  \smallskip

  \noindent Case 2) $|p_0^k|<|\Er(p_1)|$. The situation is as follows
  \begin{align*}
    \varphi(w) &=
    \overbrace{\rule{0pt}{1.9ex}\boxed{\kern1ex{p_0}\kern1ex}\,\boxed{\kern1ex{p_0}\kern1ex}
      \,\cdots\,\boxed{\kern1ex{p_0}\kern1ex}}^{k\text{-times}}\boxed{\kern8.5ex{p_1}\kern8.5ex}\,\cdots\\
    \Er(\varphi(w)) &= \boxed{\kern8.5ex{\Er(p_1)}\kern8.5ex}\,\cdots
  \end{align*}
  Thus there exists $q\in\{0,1\}^*$ such that $p_0^kq=\Er(p_1)$, i.e.,
  $p_1=\Er(q)\big(\Er(p_0)\big)^k$. As $q$ is also a prefix of $p_1$, we derive $q=\Er(q)$, i.e.,
  $q$ is an antipalindrome. Since $\Er(p_1)$ is a palindrome, we have $p_0^kq=\Rr(q)p_0^k$.  By item
  (ii) of Proposition~\ref{p:properties}, there exist $ A,B\in\{0,1\}^*$ and $N\in\N$ such that
  \begin{equation}\label{eq:eqsol}
      q=BA,\quad \Rr(q)=AB,\quad p_0^k=(AB)^NA.
  \end{equation}
  First two formulae used together, $AB=\Rr(q)=\Rr(BA)=\Rr(A)\Rr(B)$, imply that $A$ and $B$ are
  palindromes. Since $q=BA$ is an antipalindrome, we can use item (iii) of
  Proposition~\ref{p:properties} to infer that there exists $c\in\{0,1\}^*$, $c$ a palindrome, such
  that $B=c\big(\Er(c)c\big)^s$ and $A=\big(\Er(c)c\big)^t\Er(c)$ for some $s,t\in\N$. Substituting
  these forms back to~(\ref{eq:eqsol}) we obtain
  \begin{equation}\label{eq:p0p1}
    \begin{split}
    p_0^k &= \big(\Er(c)c\big)^{(t+s+1)N+t}\Er(c), \\
    p_1 &= q\big(\Er(p_0)\big)^k = \big(c\Er(c)\big)^{(t+s+1)(N+1)+t}c.
    \end{split}
  \end{equation}
  If $(t+s+1)N+t\geq 1$ then by Lemma~\ref{lem:primit} $p_0^k$ is a primitive word, and, thus $k=1$.
  Consequently, $w$ has a prefix $0^k1=01$ and the forms of $\varphi(0)$ and $\varphi(1)$ given
  by~(\ref{eq:p0p1}) are in agreement with the statement of the lemma. On the other hand, if
  $(t+s+1)N+t=0$ then $N=t=0$. It follows from~(\ref{eq:p0p1}) that $p_0^k=\Er(c)$ and thus $c=d^k$
  for some palindrome $d$, i.e., $p_0=\Er(d)$ and $p_1=\big(d^k\Er(d)^k)^{s+1}d^k$. Using
  Lemma~\ref{lem:vyluc} with $x=d$, $k_0=0$, $h_0=1$, $k_1=k(s+2)$ and $h_1=k(s+1)$, the equality
  $k_0+k_1=h_0+h_1$ implies that $k=1$. Therefore the forms of $\varphi(0)$ and $\varphi(1)$ are in
  the agreement with the statement of the lemma.

  Now consider the case when $1$ is a prefix of $\varphi(w)$. Then $0$ is a suffix of $\varphi(w)$,
  and thus $1$ is a prefix of $w$ and $0$ is a suffix of $w$. The discussion would follow
  analogically to the case where $0$ is a prefix of $\varphi(0)$, with the only exception that we
  would need to show that if $10^k$ is a suffix of $w$ then $k=1$.

  Finally, assume that neither $00$ nor $11$ is a factor of $c$. Then the palindrome $c$, which is a
  prefix of $\varphi(1)$ and starts with the letter $1$, is of the form $c=(10)^m1$ for some
  $m$. Thus $\varphi(0)=(01)^{i}0$ and $\varphi(1)=(10)^{j}1$ for some $i,j\in\N$. In such a case
  the fixed point of $\varphi$ is periodic, which is a contradiction with the assuption $(\star)$.

  \medskip
  
  Let us now prove the statement (ii). Note that in the proof of (i) we have shown that if $0$ is a
  prefix of $\varphi(w)$ then $w$ has a prefix $01$ and a suffix $1$, and $p_0$ and $p_1$ have the
  form~(\ref{eq:req_form_p0p1}).  We will now verify that the penultimate letter $a$ in $w$ is equal
  to $0$, i.e., that $w$ has 01 as a suffix.  The situation is as follows
  \begin{align*}
    \varphi(w) &= \boxed{\kern1ex{p_0}\kern1ex}\,\boxed{\kern9.7ex{p_1}\kern9.7ex}\,\cdots\\
    \Er(\varphi(w)) &= \boxed{\kern8ex{\Er(p_1)}\kern8ex}\,
    \boxed{\rule[-0.58ex]{0pt}{2.31ex}\kern1.43ex{u}\kern1.43ex}\,\cdots
  \end{align*}
  where $u$ is a word such that $p_0p_1=\Er(p_1)u$. Necessarily $|u|=|p_0|$ and it follows
  from~(\ref{eq:req_form_p0p1}) that $u=\big(c\Er(c)\big)^{l_0}c$. Moreover, $u$ is a prefix of
  $E(\varphi(a))$. Since $u$ starts with $c$, it is not a prefix of $\Er(p_1)$. Thus the penultimate
  letter $a$ in $w$ is necessarily $0$. Therefore $w=01w'01=\Theta(0)w'\Theta(0)$. We can easily
  see that $\varphi(w')$ is created from $\varphi(w)$ by removal of a prefix and suffix
  of length $|\varphi(01)|$. Thus $\varphi(w')$ is also an antipalindrome.
  The statement follows by induction on the length of $w$.

  \medskip
  
  Finally, let us prove the statement (iii).  Let $n=|v|$. By Lemma~\ref{lem:q_is_antipalindrome},
  we can find $w$ such that $\varphi(w)$ is an antipalindrome of length
  $|\varphi(w)|>r(n)\big(|\varphi(0)|+|\varphi(1)|\big)$, where $r(n)$ is the recurrence function of
  the uniformly recurrent word $\bu$. Then $|w|>r(n)$ and $v$ is, therefore, a factor of
  $w$. Consequently, $v$ is a factor of $\Theta(u)$ for some $u\in\{0,1\}^*$, since $w=\Theta(u)$ by
  (ii). The uniqueness of the decoposition $v=x\Theta(z)y$ follows from the fact that a preimage
  under the Thue-Morse morphism is unique whenever $00$ or $11$ occur.  The statement (iii) follows.
\end{proof}

\begin{lem}\label{lem:ttvar}
  Let $\varphi$ and $\bu$ satisfy condition $(\star)$.  Then there exist words $u^{(0)}$ and
  $u^{(1)}$ such that $\varphi(0)=\Theta(u^{(0)})$ and $\varphi(1)=\Theta(u^{(1)})$.
\end{lem}

\begin{proof}
  By (i) of Lemma~\ref{lem:tvar}, $\varphi(0)$ contains either $00$ or $11$ as a factor. Thus using
  (iii) of the same lemma, the word $\varphi(0)$ can be uniquely written in the form
  $\varphi(0)=x\Theta(z)y$, where $x,y\in\{0,1,\varepsilon\}$. Since by
  Lemma~\ref{lem:unifrec_closed_under_E} the language $\Lc(\bu)$ is closed under $\Er$, we have
  $00,11\in\Lc(\bu)$, and thus $\varphi(0)\varphi(0)\in\Lc(\bu)$. Again by (iii) of
  Lemma~\ref{lem:tvar}, $\varphi(0)\varphi(0)=x'\Theta(z')y'$, where $x',y'\in\{0,1,\varepsilon\}$
  and the form $x'\Theta(z')y'$ is unique. Together, we have
  \[
  x'\Theta(z')y' = \varphi(0)\varphi(0) = x\Theta(z)yx\Theta(z)y.
  \]
  We split the discussion based on possible values of $x,y$:
  \begin{itemize}
  \item
    If $x,y\neq\varepsilon$, then since $\varphi(0)$ is a palindrome, we have $x=y$, i.e.,
    $x'\Theta(z')y' = x\Theta(z)xx\Theta(z)x$. Since $xx$ cannot be the image under $\Theta$ of a
    letter, the situation is the following one (necessarily $x'=y'=\varepsilon$)
    \vspace{-\baselineskip}
    \[
    \begin{array}{P{3ex}P{3ex}P{3ex}P{3ex}P{3ex}P{3ex}P{3ex}P{3ex}P{3ex}P{3ex}P{3ex}
        P{3ex}P{3ex}P{3ex}}
      ~ & ~ & ~ & ~ & ~ & ~ & ~ & ~ & ~ & ~ & ~ & ~ & ~ & ~ \\
      \hline
      \multicolumn{1}{|c|}{x} & \multicolumn{5}{c|}{\rule[-1.3ex]{0pt}{3.9ex}\Theta(z)} &
      \multicolumn{1}{c|}{x} & \multicolumn{1}{c|}{x} & \multicolumn{5}{c|}{\Theta(z)} &
      \multicolumn{1}{c|}{x} \\
      \hline\hline
      \multicolumn{2}{|c|}{\rule[-1.3ex]{0pt}{3.9ex}\Theta(z'_1)} & \multicolumn{1}{c|}{\cdots} &
      \multicolumn{2}{c|}{\Theta(z'_{k-1})} & \multicolumn{2}{c|}{\Theta(z'_k)} &
      \multicolumn{2}{c|}{\Theta(z'_{k+1})} & \multicolumn{3}{c|}{\cdots} &
      \multicolumn{2}{c|}{\Theta(z'_{2k})} \\
      \hline
    \end{array}
    \]
    and we have two different forms of $\varphi(0)=x\Theta(z)x=\Theta(z'_{k+1}\cdots z'_{2k})$, a
    contradiction.
  \item
    If $x\neq\varepsilon$, $y=\varepsilon$ then 
    $x'\Theta(z')y' = x\Theta(z)x\Theta(z)$, 
    the situation is either
    \vspace{-\baselineskip}
    \[
    \begin{array}{P{3ex}P{3ex}P{3ex}P{3ex}P{3ex}P{3ex}P{3ex}P{3ex}P{3ex}P{3ex}P{3ex}P{3ex}}
      ~ & ~ & ~ & ~ & ~ & ~ & ~ & ~ & ~ & ~ & ~ & ~  \\
      \hline
      \multicolumn{1}{|c|}{x} & \multicolumn{5}{c|}{\rule[-1.3ex]{0pt}{3.9ex}\Theta(z)} &
      \multicolumn{1}{c|}{x} & \multicolumn{5}{c|}{\Theta(z)} \\
      \hline\hline
      \multicolumn{1}{|c|}{x'} & \multicolumn{5}{c|}{\rule[-1.3ex]{0pt}{3.9ex}\Theta(z'_1\cdots z'_k)} &
      \multicolumn{2}{c|}{\Theta(z'_{k+1})} & \multicolumn{1}{c|}{\cdots} &
      \multicolumn{2}{c|}{\Theta(z'_{2k})} & \multicolumn{1}{c|}{y'} \\
      \hline
    \end{array}
    \]
    or
    \vspace{-\baselineskip}
    \[
    \begin{array}{P{3ex}P{3ex}P{3ex}P{3ex}P{3ex}P{3ex}P{3ex}P{3ex}P{3ex}P{3ex}P{3ex}P{3ex}}
      ~ & ~ & ~ & ~ & ~ & ~ & ~ & ~ & ~ & ~ & ~ & ~ \\
      \hline
      \multicolumn{1}{|c|}{x} & \multicolumn{5}{c|}{\rule[-1.3ex]{0pt}{3.9ex}\Theta(z)} &
      \multicolumn{1}{c|}{x} & \multicolumn{5}{c|}{\Theta(z)} \\
      \hline\hline
      \multicolumn{2}{|c|}{\rule[-1.3ex]{0pt}{3.9ex}\Theta(z'_1)} & \multicolumn{1}{c|}{\cdots} &
      \multicolumn{2}{c|}{\Theta(z'_{k})} & \multicolumn{2}{c|}{\Theta(z'_{k+1})} &
      \multicolumn{2}{c|}{\Theta(z'_{k+2})} & \multicolumn{1}{c|}{\cdots} &
      \multicolumn{2}{c|}{\Theta(z'_{2k+1})} \\
      \hline
    \end{array}
    \]
    and again we have a contradiction with the uniqueness of the form of $\varphi(0)$ since
    either $\varphi(0)=x\Theta(z)=\Theta(z'_{k+1}\cdots z'_{2k})y'$ or
    $\varphi(0)=x\Theta(z)=\Theta(z'_{1}\cdots z'_{k})y''$, where $y''$ is the first letter of
    $\Theta(z'_{k+1})$.
  \item
    The case $x=\varepsilon$, $y\neq\varepsilon$ can be excluded analogously to the previous one.
  \end{itemize}
  Therefore $x=y=\varepsilon$ and $\varphi(0)=\Theta(z)$ for some $z\in\{0,1\}^*$. The same property
  of $\varphi(1)$ can be proved analogously.
\end{proof}

Now we are in position to provide the proof of Theorem~\ref{thm:mainNonuniform}.

\begin{proof}[Proof of Theorem~\ref{thm:mainNonuniform}]
  By (i) of Lemma~\ref{lem:tvar} and Lemma~\ref{lem:ttvar} there is a palindrome $c$, words
  $u^{(0)},u^{(1)}$, and $l_0,l_1\in\N$ such that
  \begin{align*}
    \varphi(0) &= \Theta(u^{(0)}) = \big(\Er(c)c\big)^{l_0}\Er(c), \\
    \varphi(1) &= \Theta(u^{(1)}) = \big(c\Er(c)\big)^{l_1}c.
  \end{align*}
  Obviously, $|\varphi(0)|=2|u^{(0)}|=(2l_0+1)|c|$. Thus $|c|$ is even and
  $\Er(c)=\Theta(u_1u_2\cdots u_n)$, where $u_1u_2\cdots u_n$ is a prefix of $u^{(0)}$ such that
  $2n=|\Er(c)|$. Let us denote $\fa:=u_1u_2\cdots u_n$. As $c$ is a palindrome, $\Er(c)$ is also a
  palindrome, and thus $\fa$ is indeed an antipalindrome. The statement of the theorem follows.
\end{proof}

%
%
%
%
%
%
%

\section{Summary and comments}

In this paper we have focused on morphisms with antipalindromic fixed points. Our aim was to define
a suitable class of morphisms and study the analogy of the well known problem of Hof, Knill and
Simon~\cite{hof-knill-simon-cmp-174}, known in the combinatorics on words as the HKS conjecture.  We
defined classes $\Aj$, $\Ad$ and demonstrated two main results for them.  The question was
previously studied by Labb\'e in~\cite{labbe-memoire}. He proves that if a uniform morphism has an
antipalindromic fixed point, then $\varphi$ or $\varphi\Theta$ is conjugated to a morphism in class
$\Ec$ defined by~\eqref{eq:classEP}. He conjectures (Conjecture 5.5 in~\cite{labbe-memoire}) that in
fact, always the latter is true. Taking into account the relation of classes $\Ec$ and $\Aj$ in
Remark~\ref{rem:EPA1}, our Theorem~\ref{thm:mainUnifrom} proves this conjecture.

%

The main question about morphisms with antipalindormic fixed points remains open. We believe the
following to be true.

\begin{conj}
  Let $\varphi$ be a primitive binary morphism with an antipalindromic fixed point $\bu$.
  Then $\varphi$ or $\varphi^2$ is conjugated to a morphism in class $\Aj\cup\Ad$.
\end{conj}

Our Theorems~\ref{thm:mainUnifrom} and~\ref{thm:mainNonuniform} confirm the conjecture provided
$\varphi$ is uniform or $\bu$ is palindromic.

Let us look at the results of our paper yet from a different aspect. Theorem~\ref{thm:mainUnifrom}
can be seen as the analogy of the result of Tan~\cite{tan-tcs-389} for uniform morphisms. We have
shown that fixed points of morphisms in class $\Ad\setminus\Aj$ are not rigid, see
Remarks~\ref{rem:1} and~\ref{rem:2}.  Let us reformulate the statements of
Theorems~\ref{thm:mainUnifrom} and~\ref{thm:mainNonuniform} with this aspect in mind.

\begin{thm}
  Let $\bu$ be a fixed point of a primitive binary morphism which contains infinitely many both
  palindromes and antipalindromes.
  \begin{itemize}
  \item
    If $\bu$ is rigid, then it is a fixed point of a uniform morphism of the form $\varphi(0)=w$,
    $\varphi(1)=E(w)$ for some $w\in\{0,1\}^*$.
  \item
    If $\bu$ is not rigid, then it is a fixed point of a morphism in the form $\psi(0)=\Theta(w)$,
    $\psi(1)=\Theta\big((\Rr(w)w)^l\Rr(w)\big)$ where $w$ is an antipalindrome and $l\in\N$.
  \end{itemize}
\end{thm}

Our study motivates a number of questions for further research.

\begin{enumerate}
\item
  Infinite words containing arbitrarily long palindromes and antipalindromes can be constructed
  using the so-called palindromic and pseudopalindromic closure, introduced by de Luca and De
  Luca~\cite{deluca-deluca-tcs-362}.  They showed the construction for the Thue-Morse word. The
  details of the construction for complementary symmetric Rote words have been described
  in~\cite{blondin-masse-paquin-tremblay-vuillon-jintsequences-16}.
  In~\cite{velka-dvorakova-integers-2018}, the authors study which of words generated by
  pseudopalindromic closure besides the Thue-Morse word $\boldsymbol{t}$ are fixed points of
  morphisms.  They conjecture that only morphisms $\varphi:\{0,1\}^*\to\{0,1\}^*$ of the form
  $$
  \varphi(0) = 0(110)^k,\qquad \varphi(1)=1(001)^k,\qquad k\in\N,k\geq 1,
  $$ generate such fixed points. Let us mention that the above morphisms belong to both class $\Pc$
  and $\Aj$. It would be interesting to clarify whether other morphisms in class $\Aj$ or $\Ad$ have
  fixed points arising by pseudopalindromic closure.


\item
  Another way of constructing infinite words is given by the mapping $S:\{0,1\}^{\N}\to
  \{0,1\}^{\N}$ defined by $S(u_0u_1u_2\cdots)=v_0v_1v_2\cdots$ where $v_i=u_i+u_{i+1} \pmod 2$ for
  $i\in\N$. If $S(\bu)$ contains infinitely many palindromes of odd length with central letter 1,
  then the word $\bu$ contains infinitely many antipalindromes.  Mapping $S$ defines the relation
  between Sturmian words and complementary symmetric Rote words. In fact, $\bu$ is a complementary
  symmetric Rote word if and only if $S(\bu)$ is a Sturmian word~\cite{rote-JNT-46}.  It would be
  interesting to determine for which $\bu$ fixed point of a primitive morphism, the image $S(\bu)$
  is also fixed by a primitive morphism and contains infinitely many palindromes with central letter
  $1$. An example of such a pair $\bu$ and $S(\bu)$ is the Thue-Morse word
  $\boldsymbol{t}=01101001\cdots$ and the fixed point $\boldsymbol{d}=1011101010111010\cdots$ of the
  period-doubling morphism $D(0)=11$, $D(1)=10$. One can easily verify that
  $S(\boldsymbol{t})=\boldsymbol{d}$.

\item
  Not all palindromic infinite words are rich in palindromes, in the sense
  of~\cite{droubay-justin-pirillo-tcs-255}. An example of an infinite word which is not rich is the
  Thue-Morse word $\boldsymbol{t}$. Nevertheless, $\boldsymbol{t}$ is generated by a morphism in
  class $\Pc$, namely $\Theta^2$. The question on which morphisms in class $\Pc$ have rich fixed
  point is not solved even for the binary case.  Partial results about morphisms generating rich
  words is given in~\cite{glen-justin-widmer-zamboni-ejc-2009}. It is an interesting question to
  distinguish morphisms of classes $\Aj\cap\Pc$, $\Ad\cap\Pc$ such that their fixed points are
  $H$-rich, where $H$ is the group of morphisms and antimorphisms generated by $\Er$ and $\Rr$,
  cf.~\cite{pelantova-starosta-dm-2013}.

\item
  The antimorphism $\Er$ defining antipalindromes in this article acts on the binary alphabet
  $\{0,1\}$. We may think of a generalization to multiletter alphabets $A$. Then one needs to
  consider a group $G$ generated by antimorphisms over the monoid $A^*$ and ask when an infnite word
  contains infinitely many $f$-palindromes for each antimorphims $f\in G$. Recall that an
  $f$-palindrome is a finite word $v\in A^*$ such that $f(v)=v$.
\end{enumerate}


\section*{Acknowledgments}

 \noindent This work was supported by the project CZ.02.1.01/0.0/0.0/16\_019/0000778 from European
 Regional Development Fund. We thank \v St\v ep\'an Starosta for valuable comments.




\end{document}